\documentclass[11pt,oneside,english,reqno]{amsart}
\usepackage[T1]{fontenc}
\usepackage[latin9]{inputenc}
\usepackage{color}
\usepackage{babel}
\usepackage{amstext}
\usepackage{amsthm}
\usepackage{setspace}
\usepackage[unicode=true,pdfusetitle,
 bookmarks=true,bookmarksnumbered=false,bookmarksopen=false,
 breaklinks=false,pdfborder={0 0 0},pdfborderstyle={},backref=false,colorlinks=true]
 {hyperref}
\usepackage{breakurl}

\makeatletter
\theoremstyle{plain}
\newtheorem{thm}{\protect\theoremname}
  \theoremstyle{plain}
  \newtheorem{lem}[thm]{\protect\lemmaname}
  \theoremstyle{plain}
  \newtheorem*{thm*}{\protect\theoremname}
\theoremstyle{remark}

\newtheorem*{qst}{Question}
\usepackage{calrsfs}
\usepackage{graphicx}
\usepackage{color}
\usepackage{perpage}
\usepackage{upquote}
\usepackage{bbm}
\usepackage[shortlabels]{enumitem}

\MakePerPage{footnote}
\gdef\SetFigFontNFSS#1#2#3#4#5{} 
\gdef\SetFigFont#1#2#3#4#5{} 
\def\clap#1{\hbox to 0pt{\hss#1\hss}}

\definecolor{myblue}{rgb}{0.09,0.32,0.44} 
\hypersetup{pdfborder={0 0 0},pdfborderstyle={},colorlinks=true,linkcolor=myblue,citecolor=myblue,urlcolor=blue}

\theoremstyle{remark}
\newtheorem*{qst*}{Question}
\newtheorem*{rmrks*}{Remarks}

\newlength{\tempindent}
\newcommand{\lazyenum}{
\setlength{\tempindent}{\parindent}
\begin{enumerate}[leftmargin=0cm,itemindent=0.7cm,labelwidth=\itemindent,labelsep=0cm,align=left,label=\arabic*)]
\setlength{\parskip}{\smallskipamount}
\setlength{\parindent}{\tempindent}
}

\makeatletter
\renewcommand{\andify}{%
  \nxandlist{\unskip, }{\unskip{} \@@and~}{\unskip{} \@@and~}}
\def\author@andify{%
  \nxandlist {\unskip ,\penalty-1 \space\ignorespaces}%
    {\unskip {} \@@and~}%
    {\unskip \penalty-2 \space \@@and~}%
}
\let\@wraptoccontribs\wraptoccontribs
\makeatother

\makeatother

\providecommand{\lemmaname}{Lemma}
\providecommand{\theoremname}{Theorem}
\providecommand{\theoremname}{Theorem}

\begin{document}

\begin{abstract}
We prove that all groups of exponential growth support non-constant positive harmonic functions.
In fact, out results hold in the more general case of strongly connected, finitely supported Markov chains invariant under some transitive group of automorphisms for which the directed balls grow exponentially.
\end{abstract}

\title[Positive harmonic functions]{Every exponential group supports a positive harmonic function}

\author{Gideon Amir}
\address{Bar-Ilan University, Ramat Gan 52900, Israel.}
\email{gidi.amir@gmail.com}

\author{Gady Kozma}
\address{The Weizmann Institute of
  Science, Rehovot 76100, Israel}
\email{gady.kozma@weizmann.ac.il}

\maketitle

\section{Introduction}

Let $G$ be a finitely generated group, and let $S$ be a set of generators. We call a function $f:G\to\mathbb{R}$ (right) harmonic if for every $g\in G$ we have
\[
f(g)=\frac{1}{|S|}\sum_{s\in S}f(gs).
\]
There is considrable interest in identifying spaces of harmonic functions with various properties and relating them to properties of the group $G$. In particular, the spaces of \emph{bounded} harmonic functions and of \emph{positive} harmonic functions were studied (these are related, respectively, to the Poisson and Martin boundaries of $G$). Our interest here is the following question: for which $G$ and $S$ does a positive, non-constant harmonic function exist?

One example where this problem is relatively understood is the case of nilpotent groups. In this case, Margulis \cite{M66} reduced the question from $G$ to $G/[G,G]$ i.e.\ from general nilpotent groups to abelian groups. When $S$ is symmetric, i.e.\ when $s\in S\Leftrightarrow s^{-1}\in S$, any positive harmonic function on an abelian groups is constant, see \cite{CD60} for a proof and some results also in the non-symmetric case (Margulis' reduction is not restricted to the symmetric case). Similarly, Hebisch and Saloff-Coste showed that when $G$ has polynomial word growth and $S$ is symmetric the only positive harmonic functions are the constants \cite[Theorem 6.2]{HSC}. Since by Gromov's theorem groups with polynomial word growth are virtually nilpotent, this result is quite close to Margulis', but the proof of \cite{HSC} is analytic in nature (and in particular does not use Gromov's theorem or any other algebraic structure result).

In the other direction, Bougerol and \'Elie showed that for a certain class of groups, close in spirit to Lie groups, exponential word growth implies that a positive non-constant harmonic function does exist \cite[Theorem 1.4]{BE95}. Here we show that most of the assumptions of \cite{BE95} are not necessary: the result holds for any finitely generated group and any finite set of generators. Even symmetry is not necessary. Let us state our result:

\begin{thm*}
Let $R_{n}$ be a strongly connected, finitely supported Markov chain
on some set $G$, invariant to some transitive group of automorphisms
of $G$. Assume the directed balls on $G$ grow exponentially. Then
there exists a non-constant positive harmonic function on $G$.
\end{thm*}

We recall that a Markov chain is called strongly connected if for any two states of the chain (here this is the same as the elements of $G$) there is a positive probability for the chain to reach from one state to the other in a finite number of steps. It is called ``finitely supported'' if from any state one may reach only finitely many states in a single step (in general this might depend on the starting state, but here due to the invariance all starting states are identical). The directed ball around $g$ of radius $r$, denoted by $B(g,r)$, is the set of all $h\in G$ such that for some $s\le r$ the probability to reach $h$ from $g$ in $s$ steps of the Markov chain is positive. Finally, a function $f:G\to\mathbb{R}$ is called harmonic to a Markov chain $R_n$ on $G$ if $f(R_n)$ is a martingale. The reader may readily verify that these definitions coincide with the definitions given above for a group, when $G$ is a group and the Markov chain is the one moving from a $g\in G$ to each of the elements of the form $gs$ for some $s\in S$ with probability $1/|S|$.


\begin{qst}Does the Grigorchuk group support a non-trivial positive harmonic function?
\end{qst}

\section{Proof}

It will be convenient to use graph notation, so whenever $R_n$ is a Markov chain on some set $G$, we will consider $G$ as a directed graph: the elements of $G$ will be called ``vertices'', and a couple $x,y\in G$ will be called a (directed) edge if the Markov chain can move from $x$ to $y$ in a single step (we denote $x\to y$). A path will be a sequence of vertices $x_1,x_2,\dotsc$ such that $x_i\to x_{i+1}$ and a geodesic is a path which is shortest possible between its end vertices (equivalently, between any two vertices in it).

Let $R_{n}$ be a Markov chain on some $G$. For a set $S$ of vertices, an $a\in S$ and a vertex $x\in\partial S$ (i.e.\ $x\not\in S$ but $\exists y\in S$ such that $y\to x$) let $\mu_{S}(a,x)$ be the probability that $R_{n}$,
when started from $a$, exits $S$ at $x$. If $a$ and $b$ are two
vertices we denote
\[
\epsilon(S)=\epsilon(S;a,b)=\max_{x}\frac{|\mu_{S}(a,x)-\mu_{S}(b,x)|}{\mu_{S}(a,x)}
\]
where the maximum is taken over all $x\in\partial S$ such that $\mu_{S}(a,x)>0$.
\begin{lem}
\label{lem:I thought everyone knew this}Let $R_{n}$ be a strongly connected Markov chain on some $G$. Then
there exists a non-trivial positive harmonic function on $G$ if and
only if there exist vertices $a$ and $b$ such that $\epsilon(S;a,b)\not\to0$
as $S\to G$.
\end{lem}
(we say that sets of vertices $S$ converge to $G$ if every vertex
of $G$ is eventually contained in $S$)
\begin{proof}
Let us start with the ``if'' part, for which we assume $a$ and
$b$ exist and we need to construct a positive harmonic function $f$.
In more details, our assumption is that there exist an $\epsilon>0$,
a sequence $S_{n}$ and vertices $x_{n}\in\partial S_{n}$ such that
\[
|\mu_{S_{n}}(a,x_{n})-\mu_{S_{n}}(b,x_{n})|>\epsilon\mu_{S_{n}}(a,x_{n}).
\]
We now define functions $f_{n}$ using
\[
f_{n}(v)=\frac{\mu_{S_{n}}(v,x_{n})}{\mu_{S_{n}}(a,x_{n})}
\]
(extend $f_{n}$ to outside $S_{n}$ arbitrarily).

We wish to take a subsequential limit of the $f_{n}$ and for this
we need to get an apriori bound. Here is where we use strong connectedness.
Indeed, if $v$ and $w$ are vertices, and if $\gamma$ is some path
from $v$ to $w$ which $R$ may take with probability $p>0$, then
for every $n$ sufficiently large such that $\gamma\subset S_{n}$
we have $\mu(v,x_{n})>p\mu(w,x_{n})$ and hence $f_{n}(v)>pf_{n}(w)$.
Since $f_{n}(a)=1$ we get that for every $v$ there exists a constant
$p(v)$ such that $|f_{n}(v)|\le p(v)$ for all $n$ suffciently large.
This means that we may take a subsequence $n_{k}$ such that $f_{n_{k}}(v)$
converges for every $v$. The resulting limit is positive, harmonic,
and nontrivial since $f_{n_{k}}(a)=1$ but $|f_{n_{k}}(b)-1|\ge\epsilon$
and both facts are preserved in the limit.

We move to the ``only if'' part. We thus assume that for all $a$
and $b$ $\epsilon(S)\to0$. Let $f$ be some positive harmonic function. Let $a$ and $b$ be vertices, and $S$ some finite set containing both.
Since $f(R_{n})$ is a martingale (by definition), the optional stopping
theorem tells us that
\[
f(a)=\sum_{x\in\partial S}\mu(a,x)f(x).
\]
Using this also for $b$ and subtracting gives
\begin{align*}
|f(a)-f(b)| & =\Big|\sum_{x\in\partial S}(\mu(a,x)-\mu(b,x))f(x)\Big|\\
 & \le\sum_{x}|\mu(a,x)-\mu(b,x)|f(x)\\
 & \le\sum_{x}\epsilon(S)\mu(a,x)f(x)=\epsilon(S)f(a).
\end{align*}
Taking $S\to G$ and using the assumption we get that $f(a)=f(b)$.
Since $a$ and $b$ were arbitrary, the function is constant.
\end{proof}
\begin{lem}
\label{lem:monotonicity}Let $R_{n}$ be a Markov chain on some $G$. Let $A\subset B$ be two finite subsets
of vertices and let $a$, $b\in A$. Then
\[
\epsilon(B)\le\epsilon(A).
\]
\end{lem}
\begin{proof}
Let $x\in\partial B$. Using the strong Markov property at the exit
time from $A$ we may write
\[
\mu_{B}(a,x)=\sum_{y\in\partial A}\mu_{A}(a,y)\mu_{B}(y,x).
\]
Writing the analogous equality for $b$ and subtracting we get
\begin{align*}
|\mu_{B}(a,x)-\mu_{B}(b,x)| & \le\sum_{y\in\partial A}|\mu_{A}(a,y)-\mu_{A}(b,y)|\mu_{B}(y,x)\\
 & \le\epsilon(A)\sum_{y\in\partial A}\mu_{A}(a,y)\mu_{B}(y,x)=\epsilon(A)\mu_{B}(a,x).
\end{align*}
Taking the maximum over $x\in\partial B$ proves the lemma.
\end{proof}
\begin{proof}[Proof of the theorem]
Assume such a function does not exist. By lemma \ref{lem:I thought everyone knew this}
we know that for all $a$ and $b$ $\, \epsilon(B(a,r);a,b)\to0$ where
$B(a,r)$ is the directed ball of radius $r$ centered at $a$, or
in other words, all vertices which can be reached from $a$ by a path
of length less than $r$ which the walker may take with positive probability
(we use here strong connectedness to claim that $B(a,r)\to G$). Let
$\delta>0$ be some parameter. Since $R_n$  has finite support there exists some $r_{0}$ sufficiently large
such that $\epsilon(B(a,s);a,b)<\delta$ for all $s>r_{0}$ and all
$a\to b$ (the value of $a$ does not matter, from transitivity).

Let now $x\in\partial B(a,r)$. Let $\gamma$ be a geodesic from $a$
to $x$ i.e.\ $a=\gamma_{0}\to\gamma_{1}\to\dotsb\to\gamma_{r}=x$. Write
\[
\mu_{B(a,r)}(a,x)=\prod_{i=0}^{r-1}\frac{\mu_{B(a,r)}(\gamma_{i},x)}{\mu_{B(a,r)}(\gamma_{i+1},x)}.
\]
(with the convention that $\mu_{B(a,r)}(x,x)=1$). Write also
\begin{equation}
\left|\frac{\mu_{B(a,r)}(\gamma_{i},x)}{\mu_{B(a,r)}(\gamma_{i+1},x)}-1\right|
\le\epsilon(B(a,r);\gamma_{i},\gamma_{i+1}).\label{eq:1}
\end{equation}
Now, $B(\gamma_{i},r-i)\subset B(a,r)$ and hence by lemma \ref{lem:monotonicity}
\begin{equation}
\epsilon(B(a,r);\gamma_{i},\gamma_{i+1})\le\epsilon(B(\gamma_{i},r-i);\gamma_{i},\gamma_{i+1}).\label{eq:2}
\end{equation}
Let $\varphi$ be an automorphism of $G$ taking $\gamma_{i}$
to $a$ and preserving the Markov chain $R$, and let $s_{i}=\varphi(\gamma_{i+1})$.
Then
\begin{equation}
\epsilon(B(\gamma_{i},r-i);\gamma_{i},\gamma_{i+1})
=\epsilon(B(a,r-i);a,s_{i}).\label{eq:3}
\end{equation}
Finally, if $i<r-r_{0}$ then
\[
\left|\frac{\mu_{B(a,r)}(\gamma_{i},x)}{\mu_{B(a,r)}(\gamma_{i+1},x)}-1\right|
\stackrel{\textrm{(\ref{eq:1},\ref{eq:2},\ref{eq:3})}}{\le}
\epsilon(B(a,r-i);a,s_{i})\le\delta.
\]
For $i\ge r-r_{0}$ we estimate simply $\mu_{B(a,r)}(\gamma_{i+1},x)\ge p\mu_{B(a,r)}(\gamma_{i},x)$
where $p$ is the minimum non-zero probability that a step of $R$
may take ($p>0$ because $R$ is finitely supported and invariant).
We get
\[
\mu_{B(a,r)}(x)\ge(1-\delta)^{r-r_{0}}p^{r_{0}}.
\]
Since $x\in\partial B(a,r)$ was arbitrary we get that
\[
|\partial B(a,r)|\le(1-\delta)^{r_{0}-r}p^{-r_{0}}.
\]
Summing over $r$ gives
\[
|B(a,s)|=\sum_{r=0}^{s-1}|\partial B(1,r)|\le\sum_{r=0}^{s-1}(1-\delta)^{r_{0}-r}p^{-r_{0}}\le\frac{p^{-r_{0}}}{\delta}(1-\delta)^{-s}.
\]
Since $\delta$ was an arbitrary positive number, we reach a contradiction
to our assumption that directed balls grow exponentially.
\end{proof}

\subsection*{Acknowledgements}
While performing this research, G.A. was supported by the Israel Science Foundation grant \#575/16 and by GIF grant \#I-1363-304.6/2016. G.K. was supported by the Israel Science Foundation grant \#1369/15 and by the Jesselson Foundation.

\end{document}